\documentclass[11pt]{article}

\usepackage[margin=1in]{geometry}
\usepackage{amsmath,amssymb,amsthm}
\usepackage{url}
\newtheorem{theorem}{Theorem}[section]
\newtheorem{lemma}[theorem]{Lemma}
\newtheorem{corollary}[theorem]{Corollary}

\newcommand{\ex}{\operatorname{ex}}

\title{Linear extremal bounds for a family of forbidden $0$-$1$ matrices}
\author{Jesse Geneson}
\date{}

\begin{document}

\maketitle

\begin{abstract}
Fulek defined the $0$-$1$ matrix
\[
L_3=\begin{pmatrix}
1&0&0&1&0\\
0&0&0&0&1\\
0&1&1&0&0
\end{pmatrix}
\]
and asked whether $\ex(n,L_3) = O(n)$. We prove that every $r\times s$ $0$-$1$ matrix avoiding $L_3$ has at most
$27r+2s$ $1$ entries.  Fulek's general lower bound construction has
$6n-8$ $1$ entries, so
\[
6n-8\leq \ex(n,L_3)\leq29n
\]
for $n\geq5$.  The same argument applies to an infinite family.  If
$Q_{a,b,k,\ell}$ is the light three-row matrix with column word
$1^a3^k1^b2^\ell$, where $a,b,\ell\geq1$ and $k\geq2$, then
\[
\ex(r,s,Q_{a,b,k,\ell})
\leq\bigl(5(k-1)(4b+1)+a+b+\ell-1\bigr)r+2s.
\]
This verifies a conjecture of Pettie and Tardos on linear light patterns
  for an infinite family that includes the previously unresolved
  weight-five pattern $L_3$. The proof assigns matrix entries to edges of a bar $1$-visibility
hypergraph, cuts gaps to control the multiplicity of these edges, and
charges the cuts to a noncrossing graph on the rows. 
\end{abstract}



\section{Introduction}

Rows of a matrix are ordered from top to bottom, and columns are ordered
from left to right.  Let $M$ be a $0$-$1$ matrix and let $Q$ be a
$k\times \ell$ $0$-$1$ matrix.  We say that $M$ \emph{contains} $Q$ if
there are row indices
\[
i_1<\cdots<i_k
\]
and column indices
\[
j_1<\cdots<j_\ell
\]
such that
\[
M_{i_a,j_b}=1
\qquad\text{whenever }Q_{a,b}=1.
\]
The entries selected at positions where $Q$ has a $0$ are unrestricted.
If $M$ does not contain $Q$, then $M$ \emph{avoids} $Q$.  Write $w(M)$
for the number of $1$ entries in $M$; this number is the \emph{weight} of
$M$.  Let $\ex(r,s,Q)$ be the largest
weight of an $r\times s$ $0$-$1$ matrix avoiding $Q$, and put
\[
\ex(n,Q)=\ex(n,n,Q).
\]

Fulek \cite{Fulek2009} proved linear bounds for two patterns denoted by
$L_1$ and $L_2$, denoted two further patterns by $L_3$ and $L_4$, and
proposed them as candidates for an extension of his visibility method.  In
this note, $L_3$ always denotes Fulek's matrix
\begin{equation}\label{eq:L3}
L_3=\begin{pmatrix}
1&0&0&1&0\\
0&0&0&0&1\\
0&1&1&0&0
\end{pmatrix}.
\end{equation}
In 2014 we extended Fulek's method from bar visibility
graphs to bar visibility hypergraphs \cite{GenesonShen2015}.  For every
$s\geq0$, that paper defined the bar $s$-visibility hypergraph of a
finite family of disjoint horizontal bars and proved that such a
hypergraph on $v$ bars has at most $(2s+3)v$ edges.
Lemma~\ref{lem:bar-visibility} below is the case $s=1$ of that bound.

The paper \cite{GenesonShen2015} used the notation $L_3$ for the
different matrix
\[
\begin{pmatrix}
0&1&1&1&0\\
1&0&0&0&0\\
0&0&0&0&1\\
0&0&1&0&0
\end{pmatrix}
\]
and proved a linear bound for that matrix.  That notation is unrelated
to the notation in \eqref{eq:L3}.  No theorem of \cite{GenesonShen2015}
states a bound for the matrix in \eqref{eq:L3}.

Fulek also considered
\[
L_5=\begin{pmatrix}
1&0&0&0&0\\
0&0&0&1&0\\
0&0&0&0&1\\
0&1&1&0&0
\end{pmatrix}.
\]
In 2009 I proved that every double permutation matrix has a linear
extremal function \cite{Geneson2009}.  Apply that theorem to the double
permutation matrix $D$ obtained by duplicating every column of the
$4\times4$ permutation matrix whose columns, from left to right, have
their $1$ entries in rows $1,4,2,3$,
\[
D=\begin{pmatrix}
1&1&0&0&0&0&0&0\\
0&0&0&0&1&1&0&0\\
0&0&0&0&0&0&1&1\\
0&0&1&1&0&0&0&0
\end{pmatrix}.
\]
The matrix obtained from $D$ by selecting columns $1,3,4,5,7$ is $L_5$.
Every $0$-$1$ matrix avoiding $L_5$ therefore avoids $D$, and hence
\[
\ex(n,L_5)\leq \ex(n,D)=O(n).
\]
Thus the linear bound for $L_5$ follows from \cite{Geneson2009},
independently of the result proved here.  Fulek observed that the
linearity of $\ex(n,L_3)$ would also yield a linear bound on
$\ex(n,L_5)$ \cite{Fulek2009}, so Theorem~\ref{thm:main} recovers the
same conclusion by that route.

Our first result determines the order of magnitude of $\ex(n,L_3)$.

\begin{theorem}\label{thm:main}
For every $n\geq 5$,
\[
6n-8\leq \ex(n,L_3)\leq 29n.
\]
Consequently, $\ex(n,L_3)=\Theta(n)$.
\end{theorem}

The upper bound is new.  The lower bound is Fulek's general
construction, Proposition~1 of \cite{Fulek2009}.

The argument proves more.  A $0$-$1$ matrix is \emph{light} if every column
contains
at most one $1$ entry.  In a light matrix with no zero column, every column
contains exactly one $1$ entry.  The \emph{column word} of such a matrix
records, for each column from left to right, the row containing the $1$
entry of that column.  For positive integers $a,b,k,\ell$, let
$Q_{a,b,k,\ell}$ be the light three-row matrix with column word
\begin{equation}\label{eq:family-word}
1^a3^k1^b2^\ell.
\end{equation}
Thus
\[
Q_{1,1,2,1}=L_3.
\]

\begin{theorem}\label{thm:family}
Let $a,b,\ell\geq1$ and $k\geq2$.  Every $r\times s$ $0$-$1$ matrix $M$
that avoids $Q_{a,b,k,\ell}$ satisfies
\[
w(M)\leq
\bigl(5(k-1)(4b+1)+a+b+\ell-1\bigr)r+2s.
\]
Consequently,
\[
\ex(n,Q_{a,b,k,\ell})=\Theta(n).
\]
\end{theorem}

Theorem~\ref{thm:family} gives the rectangular estimate
$\ex(r,s,L_3)\leq27r+2s$ when $a=b=\ell=1$ and $k=2$.

Pettie and Tardos \cite[Conjecture~4.2]{PettieTardos2025} considered
  the light patterns
  \[
  Q_3=
  \begin{pmatrix}
  1&0&1&0\\
  0&1&0&1
  \end{pmatrix},
  \qquad
  Q_3'=
  \begin{pmatrix}
  1&0&0&0\\
  0&0&1&0\\
  0&1&0&1
  \end{pmatrix}.
  \]
  Let $\mathcal Q_{\mathrm{PT}}$ be the set of distinct patterns obtained
  from $Q_3$ and $Q_3'$ by reflecting rows, columns, or both.  Thus
  $\mathcal Q_{\mathrm{PT}}$ consists of six patterns: its two two-row
  members have column words
  \[
  1212,\qquad 2121,
  \]
  and its four three-row members have column words
  \[
  1323,\qquad 3231,\qquad 3121,\qquad 1213.
  \]
  Call a pattern $\mathcal Q_{\mathrm{PT}}$-free if it contains no member
  of $\mathcal Q_{\mathrm{PT}}$.  Pettie and Tardos conjectured that every
  light $\mathcal Q_{\mathrm{PT}}$-free pattern has a linear extremal
  function.  They noted that the conjecture remained unconfirmed for
  some light patterns of weight five.

  Every $Q_{a,b,k,\ell}$ is $\mathcal Q_{\mathrm{PT}}$-free.  Indeed, on
  any two of its rows its column word has no alternating subsequence of
  length four, so it contains neither of the two-row members.  Moreover,
  none of the four three-row words displayed above is a subsequence of
  $1^a3^k1^b2^\ell$.  Thus Theorem~\ref{thm:family} verifies the
  Pettie--Tardos conjecture for an infinite family, including $L_3$, a
  previously unresolved pattern of weight five.

Section~\ref{sec:family} states the edge bound of
\cite{GenesonShen2015} for bar $1$-visibility hypergraphs and proves
Theorem~\ref{thm:family}. 

\section{The family bound}\label{sec:family}

Bar $s$-visibility hypergraphs were introduced in
\cite{GenesonShen2015}.  We restrict to the case $s=1$.  A \emph{bar} is a
nonempty closed horizontal line segment in the plane.  A
bar is allowed to consist of one point.  Let $\mathcal B$ be a finite family
of pairwise disjoint bars.  Three bars of $\mathcal B$ form a \emph{bar
$1$-visibility edge}
if a vertical line segment meets those three bars and no other member of
$\mathcal B$.  Let $H_1(\mathcal B)$ be the $3$-uniform hypergraph whose
vertices are the bars and whose edges are the bar $1$-visibility edges.

\begin{lemma}[Geneson--Shen \cite{GenesonShen2015}]\label{lem:bar-visibility}
If $\mathcal B$ consists of $v$ bars, then $H_1(\mathcal B)$ has at most
$5v$ edges.
\end{lemma}

Lemma~\ref{lem:bar-visibility} is the case $s=1$ of the bound $(2s+3)v$
proved in \cite{GenesonShen2015}.   Wang
\cite{Wang2026} recently determined the exact maximum number of edges
in a bar $s$-visibility hypergraph on $v$ vertices: it is
$(2s+3)v-3s^2-8s-6$ for $v\geq 2s+3$, which is $5v-17$ in the case
$s=1$.  The estimate $5v$ suffices for our purposes.

For the remainder of this section, fix $a,b,\ell\geq1$ and $k\geq2$, and
let $M$ be an $r\times s$ $0$-$1$ matrix that avoids $Q_{a,b,k,\ell}$.

\subsection{Trimming and the bar representation}

Put
\[
t=b+\ell-1.
\]
The $1$ entries of $M$ are called \emph{original}.  In each row of $M$,
delete every $1$ entry that is among the first $a$ $1$
entries of its row or among the last $t$ $1$ entries of its row.
Let $A$ be the resulting matrix.  At most $(a+t)r$ entries are deleted.

Let $S$ be the set of $1$ entries of $A$ other than the top two $1$ entries
in each column.  If a column has fewer than two $1$ entries, exclude all of
them from $S$.  The excluded entries remain in $A$.  At most $2s$ entries
are excluded, so
\begin{equation}\label{eq:initial-accounting}
w(M)\leq |S|+(a+b+\ell-1)r+2s.
\end{equation}

Assign to each row $i$ a height $h_i$, with $h_i>h_j$ whenever $i<j$.
Identify the $1$ entry of $A$ in row $i$ and column $j$ with the point whose
horizontal coordinate is $j$ and whose height is $h_i$.  Within one row, we
also identify a $1$ entry with its column, so that inequalities between $1$
entries of one row and columns compare columns.  For every nonempty row $i$
of $A$, draw the closed horizontal bar at height $h_i$ joining its
first and last $1$ entries; this bar is the \emph{initial bar} of
row $i$.  We shall remove open intervals from the bars.
The connected components of the resulting sets are closed intervals; they
are called components, and their current
family is denoted by $\mathcal B$.  A component contained in the initial bar
of row $i$ is called a row-$i$ component.  The \emph{horizontal projection}
of a component is the set of horizontal coordinates of its points; the
vertical line through column $q$ meets a component if and only if $q$ lies
in its horizontal projection.  All row-$i$ components lie at height $h_i$,
so a row-$i$ component is determined by its horizontal projection.  When the
row is fixed, we write a row-$i$ component as its horizontal projection;
thus $\{u\}$ denotes the one-point row-$i$ component at column $u$.  A
row-$i$ component \emph{meets} an interval of columns if its horizontal
projection intersects that interval.  Initially, $|\mathcal B|\leq r$.

No removal below deletes a $1$ entry of $A$.  Consequently, throughout the
construction every $1$ entry of $A$ lies on exactly one component.  The
components in one row are pairwise disjoint closed intervals at one
height, so a vertical line meets at most one component of each row, and
distinct components met by one vertical line lie in distinct rows.  Let
$(z,q)\in S$.  The top two $1$ entries of $A$ in column $q$ lie in rows
above row $z$ and lie on components, so at least two components meet the
vertical line through column $q$ at heights above $(z,q)$.

For each $(z,q)\in S$, let $C_z$ be the component containing $(z,q)$, let
$C_y$ be the lowest component that meets the vertical line through column
$q$ above $(z,q)$, and let $C_x$ be the second lowest.  Let $x$ and $y$ be
the rows of $C_x$ and $C_y$.  Then $x<y<z$.  The vertical segment from
$(z,q)$ to $C_x$ meets exactly the components $C_x$, $C_y$, and $C_z$:
any other component meeting this segment would lie above $(z,q)$ and at a
height at most that of $C_x$, contradicting the choice of $C_y$ and
$C_x$.  Hence $\{C_x,C_y,C_z\}$ is an edge of $H_1(\mathcal B)$, and we
assign $(z,q)$ to this edge.  Two points of $S$ assigned to one edge lie
on the lowest component of that edge, hence in one row and in distinct
columns.

For each row $i$, let $E_i$ be the column of its $\ell$th original $1$
entry from the right.  Set $E_i=-\infty$ if row $i$ has fewer than $\ell$
original $1$ entries.

\subsection{The gap forced by a multiple edge}

\begin{lemma}\label{lem:localization}
Suppose that $k$ points of $S$ in columns
\[
q_1<\cdots<q_k
\]
are assigned to a single edge of $H_1(\mathcal B)$, and let $C_x$, $C_y$,
and $C_z$ be the components of this edge in rows $x<y<z$.  Then row $x$ has
at least $b$ original $1$ entries after $q_k$; let $d_1<\cdots<d_b$ be the
first $b$ of them.  The last original row-$x$ $1$ entry at or before $q_k$
and the first original row-$x$ $1$ entry after $q_k$ exist; call them $u$
and $v$.  These two entries are consecutive, $u$ is a $1$ entry of $A$, and
\begin{equation}\label{eq:local-chain}
q_k\in[u,v),
\qquad
u<E_y\leq d_b\leq E_x.
\end{equation}
Moreover,
\begin{equation}\label{eq:intermediate-endpoints}
E_w\leq d_b
\qquad\text{for every }x<w<z.
\end{equation}
\end{lemma}

\begin{proof}
Each of the $k$ points lies on the lowest component of the edge, which is
$C_z$, so each point is a $1$ entry of $A$ in row $z$.  For each $i$, the
vertical segment that assigns the point in column $q_i$ to the edge meets
$C_x$, $C_y$, and $C_z$.  Hence the vertical line through $q_i$ meets all
three components, and $q_i$ lies in the horizontal projection of each of
them.

The component $C_y$ is contained in the initial bar of row $y$.  Write the
original row-$y$ $1$ entries as $f_1<\cdots<f_h$.  Since this initial bar is
nonempty, it is $[f_{a+1},f_{h-t}]$.  The column $q_k$ lies in the
horizontal projection of $C_y$, so
\[
q_k\leq f_{h-t}<f_{h-\ell+1}=E_y,
\]
where the strict inequality follows from $t=b+\ell-1$ and $b\geq1$.
Therefore
\begin{equation}\label{eq:q-before-Ey}
q_k<E_y.
\end{equation}
Write the original row-$x$ $1$ entries as
\[
c_1<\cdots<c_m.
\]
The component $C_x$ is contained in the initial bar $[c_{a+1},c_{m-t}]$ of
row $x$, and $q_1$ and $q_k$ lie in the horizontal projection of $C_x$.
Hence $c_{a+1}\leq q_1$, so the first $a$ original row-$x$ $1$
entries precede $q_1$, and $q_k\leq c_{m-t}$.  Since $t\geq b$, at least
$b$ original row-$x$ $1$ entries follow $q_k$.  Thus
$d_1,\ldots,d_b$ exist.

If $d_b<E_y$, use rows $x,y,z$.  In row $x$, select its first $a$ original
$1$ entries and the entries in columns $d_1,\ldots,d_b$.  In row $z$,
select the entries in columns $q_1,\ldots,q_k$.  In row $y$, select its last
$\ell$ original $1$ entries.  Their columns occur in the order
\[
1^a3^k1^b2^\ell,
\]
so they form $Q_{a,b,k,\ell}$.  This contradicts the choice of $M$.
Therefore
\begin{equation}\label{eq:db-after-Ey}
E_y\leq d_b.
\end{equation}

Since $c_{a+1}\leq q_1\leq q_k$, the last original row-$x$ $1$ entry at or
before $q_k$ exists; write it as $u=c_j$.  The bounds
$c_{a+1}\leq q_k\leq c_{m-t}$ give
\[
a+1\leq j\leq m-t.
\]
Thus $u$ belongs to $A$, the entry $v=c_{j+1}$ exists, and $u,v$ are
consecutive with $q_k\in[u,v)$.  Equation~\eqref{eq:q-before-Ey} gives
$u<E_y$.  Since $v$ is the first original $1$ entry after $q_k$, we have
$d_b=c_{j+b}$.  Therefore
\[
d_b=c_{j+b}\leq c_{m-t+b}=c_{m-\ell+1}=E_x.
\]
Together with \eqref{eq:q-before-Ey} and \eqref{eq:db-after-Ey}, this proves
\eqref{eq:local-chain}.

Finally, let $x<w<z$.  If $E_w>d_b$, use the same row-$x$ and row-$z$
entries as above and the last $\ell$ original row-$w$ $1$ entries.  They
again occur in the order $1^a3^k1^b2^\ell$, producing
$Q_{a,b,k,\ell}$.  Thus $E_w\leq d_b$.
\end{proof}

\subsection{Cutting the gaps}

A \emph{gap} of row $i$ is an open interval whose endpoints are consecutive
original row-$i$ $1$ entries.  A gap carries its row: gaps of distinct rows
are distinct objects, even when their intervals coincide.  A \emph{cut}
removes one gap of one row from the components of that row and, in some cases, one further open interval, the
\emph{auxiliary notch}, described below; the two removals together form one
cut.  The removed gap is the \emph{main gap} of the cut, and a gap
is \emph{used} when it is the main gap of a cut.

Suppose that some edge of $H_1(\mathcal B)$ has at least $k$ assigned
points.  These points lie in distinct columns.  Apply
Lemma~\ref{lem:localization} to the $k$ assigned points with the smallest
columns.  The lemma produces consecutive original row-$x$ $1$ entries $u$
and $v$, so $(u,v)$ is a gap of row $x$; it is the main gap of the present
cut, which is made in row $x$.  Inductively, every used gap $(u',v')$ of a
row $x'$ satisfies the following property: no current row-$x'$ component
meets the open interval $(u',v')$, and $\{u'\}$ is a current row-$x'$
component.
This property is established below at the cut using each gap and shown
to persist under all later cuts.  The present main gap has not been used
before.  If it had been used, then no current row-$x$ component would meet
$(u,v)$, and the only row-$x$ component meeting the vertical line through
$u$ would be $\{u\}$.  Since $q_k\in[u,v)$ and $q_1<q_k<v$, no single
row-$x$ component would meet the vertical lines through both $q_1$ and
$q_k$.  This contradicts the fact that both of these vertical lines meet
$C_x$.

Remove the open interval $(u,v)$ from the current row-$x$ components.  If
the row-$x$ component containing $u$ also contains a point to the left of
$u$, let $f$ be the preceding original row-$x$ $1$ entry.  The entry at $u$
belongs to $A$ by Lemma~\ref{lem:localization}, so $f$ exists.  Choose
$\varepsilon>0$ such that
\[
\varepsilon<u-f
\qquad\text{and}\qquad
[u-\varepsilon,u]
\text{ lies in the row-$x$ component that contained }u,
\]
and remove $(u-\varepsilon,u)$ from that component.  This is the auxiliary
notch.  If the row-$x$ component containing $u$ has no point to the left
of $u$, no auxiliary notch is made.  In either case, after the main gap and the
possible auxiliary notch are removed, $\{u\}$ is a current row-$x$
component and no row-$x$ component meets $(u,v)$.

No $1$ entry of $A$ is removed: both removals delete only points of
row-$x$ components, the main gap contains no original row-$x$
$1$ entry, and the auxiliary notch lies between the consecutive original
row-$x$ entries $f$ and $u$.  Immediately before the cut, at most one
row-$x$ component meets the main gap.  To verify this, observe that the only earlier removal from
row $x$ that can lie properly inside $(u,v)$ is an auxiliary notch
immediately to the left of $v$: the endpoints of every gap of row $x$ are
original row-$x$ $1$ entries, every auxiliary notch in row $x$ adjoins an
original row-$x$ $1$ entry on its right, and there is no original row-$x$
$1$ entry strictly between $u$ and $v$.  Such a notch
leaves at most the part adjoining $u$ in a row-$x$ component.  Removing the
main gap therefore increases the number of components by at most one.  The
auxiliary notch splits at most one component.  Each cut increases
$|\mathcal B|$ by at most two.

After the cut, the $k$ chosen points are no longer assigned to one common
edge whose components lie in rows $x$, $y$, and $z$: such an edge would
require one row-$x$ component to meet the vertical lines through all of
$q_1,\ldots,q_k$.  If $u<q_k<v$, then no row-$x$ component meets the
vertical line through $q_k$.  If $q_k=u$, then the only row-$x$ component
meeting the vertical line through $q_k$ is $\{u\}$, and $\{u\}$ does not
meet the vertical line through $q_1<u$.  Since $\{u\}$ is a current
row-$x$ component and no row-$x$ component meets $(u,v)$, the inductive
property holds for the present main gap at the moment of its cut.  The
property persists under every later cut.  Consider a used gap $(u',v')$ of
a row $x'$.  Cuts only remove points, so every later row-$x'$ component is
contained in a row-$x'$ component current immediately after the cut that
used $(u',v')$; no later row-$x'$ component therefore meets $(u',v')$.
Moreover, no cut removes a $1$ entry of $A$, and $u'$ is a $1$ entry of $A$
by Lemma~\ref{lem:localization}, so $u'$ always lies on a row-$x'$
component; this component is contained in the row-$x'$ component $\{u'\}$
current immediately after the cut that used $(u',v')$, hence equals it.
After the cut, recompute all assignments.

Continue until every edge of $H_1(\mathcal B)$ has at most $k-1$ assigned
points.  A gap cannot be used twice by the preceding freshness
argument, and each row has only finitely many gaps.  The process therefore
terminates: every cut uses a gap that no earlier cut used, and there are
finitely many gaps in total.

\subsection{Counting the cuts}

Let $K$ be the number of cuts.  Consider one cut, made for an edge with
components in rows $x<y<z$, and keep the notation of
Lemma~\ref{lem:localization}.  Let $p$ be the
first row attaining
\[
\max\{E_w:x<w<z\}.
\]
This maximum is finite because the set contains $E_y$, which is finite by
\eqref{eq:local-chain}.  Equations
\eqref{eq:local-chain} and \eqref{eq:intermediate-endpoints} give
\begin{equation}\label{eq:cut-chain}
u<E_y\leq E_p\leq d_b\leq E_x.
\end{equation}
The choice of the first maximizing row also gives
\begin{equation}\label{eq:first-max}
E_j<E_p
\qquad\text{for every }x<j<p.
\end{equation}
Associate the pair $(x,p)$ with the cut.

Let $G$ be the simple graph on the rows $1,\ldots,r$ of $M$ whose edges are
the distinct associated pairs.  The graph $G$ is noncrossing in
the row order.  If two of its edges had alternating endpoints
\[
x<x'<p<p',
\]
then \eqref{eq:first-max} for $(x,p)$ would give $E_{x'}<E_p$.
Equation~\eqref{eq:first-max} for $(x',p')$ would give $E_p<E_{p'}$, and
\eqref{eq:cut-chain} for $(x',p')$ would give $E_{p'}\leq E_{x'}$.  Hence
\[
E_{x'}<E_p<E_{p'}\leq E_{x'},
\]
a contradiction.

The graph $G$ is outerplanar.  Draw the rows $1,\ldots,r$ as points on a
horizontal line and each edge $\{i,j\}$ of $G$ as the semicircle above the
line with endpoints $i$ and $j$.  Two such semicircles cross if and only if
their endpoints alternate in the row order, and no two edges of $G$ have
alternating endpoints, so this drawing is a planar embedding with every
vertex on the outer face.  Every simple outerplanar graph on $r\geq2$
vertices has at most $2r-3$ edges: it is a subgraph of a maximal
outerplanar graph on the same vertex set, and every maximal outerplanar
graph on $r\geq2$ vertices has exactly $2r-3$ edges, by Corollary~11.9(a)
of \cite{Harary1969}.  Hence $G$ has at most $2r-3$ edges when $r\geq2$.

One associated pair occurs for at most $b$ cuts.  Fix $(x,p)$ and write the
original row-$x$ $1$ entries as $c_1<\cdots<c_m$.  If the main gap of a cut
with associated pair $(x,p)$ is $(c_j,c_{j+1})$, then $d_b=c_{j+b}$, and
\eqref{eq:cut-chain} gives
\begin{equation}\label{eq:gap-multiplicity}
c_j<E_p\leq c_{j+b}.
\end{equation}
The pair $(x,p)$ is the associated pair of at least one cut, so some index
$j$ satisfies \eqref{eq:gap-multiplicity}; for that index,
$E_p\leq c_{j+b}\leq c_m$.  Let $j_0$ be the least index such that
$E_p\leq c_{j_0}$.  Every index $j$
satisfying \eqref{eq:gap-multiplicity} belongs to
\[
\{j_0-b,\ldots,j_0-1\}.
\]
There are at most $b$ such indices, and different cuts use different main
gaps.  Therefore
\begin{equation}\label{eq:number-cuts}
K\leq b(2r-3)
\end{equation}
when $r\geq3$.

\subsection{Completion of the proof}

Let $V$ be the final number of components.  If $r\geq3$, then
\eqref{eq:number-cuts} gives
\[
V\leq r+2K
\leq r+2b(2r-3)
\leq(4b+1)r.
\]
Every point of $S$ is assigned to an edge of the final hypergraph
$H_1(\mathcal B)$, and every edge has at most $k-1$ assigned points.
Lemma~\ref{lem:bar-visibility} therefore gives
\[
|S|\leq(k-1)|E(H_1(\mathcal B))|
\leq5(k-1)V
\leq5(k-1)(4b+1)r.
\]
Combining this estimate with \eqref{eq:initial-accounting} proves
\[
w(M)\leq
\bigl(5(k-1)(4b+1)+a+b+\ell-1\bigr)r+2s.
\]
If $r\leq2$, then $w(M)\leq rs\leq2s$, so the same estimate holds.

For the square lower bound, consider the $n\times n$ matrix whose first row
has every entry equal to $1$ and whose other entries are $0$.  It has
weight $n$, and it avoids every
pattern with $1$ entries in at least two rows, since a containment would
place $1$ entries of the matrix in two distinct rows.  Because $a,\ell\geq1$
and $k\geq2$, the pattern $Q_{a,b,k,\ell}$ has a $1$ entry in each of its
three rows, so the matrix avoids $Q_{a,b,k,\ell}$.  This proves
$\ex(n,Q_{a,b,k,\ell})=\Theta(n)$ and completes the proof of
Theorem~\ref{thm:family}.

\begin{proof}[Proof of Theorem~\ref{thm:main}]
Theorem~\ref{thm:family} with $a=b=\ell=1$, $k=2$, and $r=s=n$ gives
\[
\ex(n,L_3)\leq 29n.
\]
For $n\geq\max\{e-1,f-1\}$, Proposition~1 of \cite{Fulek2009} states that
every pattern $P$ with $e$ rows, $f$ columns, and at least one $1$ entry
satisfies
\[
\ex(n,P)\geq n(e+f-2)-(e-1)(f-1).
\]
The matrix $L_3$ has $e=3$ and $f=5$, so $\ex(n,L_3)\geq6n-8$.
\end{proof}

The smallest member of the family in Theorem~\ref{thm:family} that is not
obtained by setting $b=1$ is
\[
Q_{1,2,2,1}=
\begin{pmatrix}
1&0&0&1&1&0\\
0&0&0&0&0&1\\
0&1&1&0&0&0
\end{pmatrix},
\]
whose column word is $133112$.  Substitution in
Theorem~\ref{thm:family} gives the following explicit bound.

\begin{corollary}\label{cor:second-core}
Every $r\times s$ $0$-$1$ matrix $M$ that avoids $Q_{1,2,2,1}$ satisfies
\[
w(M)\leq48r+2s.
\]
In particular,
\[
\ex(n,Q_{1,2,2,1})\leq50n.
\]
\end{corollary}

Fulek's pattern $L_4$ is
\[
L_4=\begin{pmatrix}
0&1&0&0&1&0\\
1&0&0&0&0&1\\
0&0&1&1&0&0
\end{pmatrix}.
\]
Its column word is $213312$.  Fulek proposed $L_4$ together with $L_3$
and noted that the extremal functions of both are in $O(n\alpha(n))$,
where $\alpha$ is the inverse Ackermann function \cite{Fulek2009}. The method of this note does not prove a linear bound for $L_4$. 

\section*{Acknowledgments}

Codex with GPT-5.6 and Claude Code with Fable 5 were used
for proof exploration, proof criticism, exposition, and revision.

\end{document}